\documentclass[12pt]{amsart} 

\usepackage{amsmath} \usepackage{amssymb} \usepackage{amsthm}
\usepackage{amscd} 
\usepackage{color} 
\usepackage[all]{xy} 
\usepackage{enumerate}
\usepackage{mathrsfs}
\usepackage{multirow}
\usepackage{bigdelim}
\usepackage[bookmarks=false]{hyperref}

\def\L{{\mathcal{L}}} 
\def\M{{\mathcal{M}}}
\def\o{{\omega}}
\def\O{{\mathcal{O}}}
\def\Z{{\mathbb{Z}}}

\def\Spec{{\mathrm{Spec\; }}} 

\def\Pic{{\mathrm{Pic}}}
\newtheorem{thm}{Theorem}[section] 

\newtheorem{prop}[thm]{Proposition}

\newtheorem{lem}[thm]{Lemma}
\theoremstyle{definition}

\newtheorem{eg}[thm]{Example} 
\theoremstyle{remark}
\newtheorem{rem}[thm]{Remark}

\newtheorem*{acknowledgement}{Acknowledgments}

\title{A characterization of ordinary abelian varieties by the Frobenius push-forward \\ of the structure sheaf I\hspace{-1pt}I}
\author{Sho Ejiri and Akiyoshi Sannai}
\address{Graduate School of Mathematical Sciences, the University of Tokyo, 3-8-1 Komaba, Meguro-ku, Tokyo 153-8914, Japan.}
\email{ejiri@ms.u-tokyo.ac.jp}
\address{Research Institute for Mathematical Sciences, Kyoto University, Kyoto
606-8502, Japan}
\email{sannai@kurims.kyoto-u.ac.jp}

\begin{document}
\tolerance = 9999
\maketitle
\markboth{SHO EJIRI AND AKIYOSHI SANNAI}{FROBENIUS PUSH-FORWARD OF THE STRUCTURE SHEAF I\hspace{-1pt}I} 
\begin{abstract}
In this paper, we prove that a smooth projective variety $X$ of characteristic $p>0$ is an ordinary abelian variety if and only if 
$K_X$ is pseudo-effective and 
$F_*^e\O_X$ splits into a direct sum of line bundles for an integer $e$ with $p^e>2$.
\end{abstract}
\section{Introduction}\label{section:intro}
We fix an algebraically closed field $k$ of characteristic $p>0$ and consider a projective variety $X$ over $k$.
The $\O_X$-module structure of the push-forward $F_*\O_X$ of the structure sheaf $\O_X$ along the Frobenius map $F$ on $X$ 
tells us much information about $X$.
Kunz proved that the smoothness of $X$ is equivalent to the local freeness of $F_*\O_X$~\cite{Kun76}. 
Mehta and Ramanathan showed that an $F$-split variety, 
a variety $X$ with splitting of the natural morphism $\O_X\to F_*\O_X$, 
has some pleasant properties~\cite{MR85}.
It is also known that, in some low-dimensional cases, 
the direct summands of $F_*\O_X$ generate the bounded derived category of $X$ 
\cite{Har15,OU13}. 
The second author and Tanaka found the following characterization of ordinary abelian varieties:
\begin{thm}[\textup{\cite[Theorem 1.1]{ST16}}]\label{thm:ST}
Let $X$ be a smooth projective variety over an algebraically closed field of characteristic $p>0$. 
Then $X$ is an ordinary abelian variety if and only if the following two conditions hold:
\begin{itemize} 
\item $F^e_*\O_X$ is isomorphic to a direct sum of line bundles for infinitely many $e>0$.
\item $K_X$ is pseudo-effective.
\end{itemize}
\end{thm}
It is natural to ask whether we can characterize ordinary abelian varieties by using only $F_*\O_X$.
The answer is no. 
Indeed, as described in Section~\ref{section:surf}, 
a non-abelian smooth projective surface $S$ of characteristic 2, which is called Igusa surface, 
has the property that $K_S$ is trivial and $F_*\O_S$ splits into a direct sum of line bundles.
However, we still have the same characterization by using only $F_*^e\O_X$ for $e=1$ or $2$.
The main theorem of this paper is the following:
\begin{thm}
\label{thm:ch_ab_intro} 
Let $X$ be a smooth projective variety over an algebraically closed field of characteristic $p>2$ $($resp. $p>0)$. 
Then $X$ is an ordinary abelian variety if and only if the following conditions hold: 
\begin{itemize}
\item $F_*\O_X$ $($resp. $F^2_*\O_X)$ is isomorphic to a direct sum of line bundles. 
\item $K_X$ is pseudo-effective.
\end{itemize}
\end{thm}
In order to compare with the proof of this theorem, we briefly recall that of Theorem~\ref{thm:ST}. 
We first show that each direct summand of $F^e_*\O_X$ is a $p^e$-torsion line bundle. 
Using these line bundles, we study the Albanese map $a_X$ of $X$. 
The assumption of infiniteness of $e$ ensures that the $p^e$-torsion line bundles are algebraically equivalent to zero. 
From this we obtain that $a_X$ is generically finite, and then show that it is an isomorphism. \par
In the proof of Theorem~\ref{thm:ch_ab_intro}, we also obtain $p$-torsion line bundles in the same way as above, 
but we use them to study the cotangent bundle $\Omega_X^1$ of $X$. 
We then obtain the following:
%
\begin{thm}
\label{thm:ch_triv_intro}
Let $X$ be a smooth projective variety over an algebraically closed field of characteristic $p>0$. 
Assume that the following conditions hold: 
\begin{itemize}
\item $F_*\O_X$ is isomorphic to a direct sum of line bundles. 
\item $K_X$ is pseudo-effective.
\end{itemize}
Then $\Omega_X^1$ is trivial.
\end{thm}
The key to the proof of this theorem is to associate a $p$-torsion line bundle with a logarithmic 1-form.
Let $L$ be a divisor on $X$ with $pL\sim0$ and 
$g$ be a rational function satisfying $(g)=pL$.
Then, by local calculation, we see that $dg/g$ is a global 1-form. 
We show that such global 1-forms generate $\Omega_X^1$ and form a basis of $\Omega_{K(X)/k}^1$, 
which means that $\Omega_X^1$ is trivial. \par
%
%
We conclude the introduction with an outline of the proof of Theorem~\ref{thm:ch_ab_intro}. 
Let $X$ be a non-abelian smooth projective variety with the property that 
$K_X$ is pseudo-effective and $F^e_*\O_X$ is a direct sum of line bundles.
As mentioned in Remark~\ref{rem:e to e-1}, $F_*\O_X$ is also a direct sum of line bundles,
so $\Omega_X^1$ is trivial by Theorem~\ref{thm:ch_triv_intro}. 
Then, we can use a result of Mehta and Srinivas (Theorem~\ref{thm:MS}), 
and obtain an ordinary abelian variety $A$ and a Galois \'etale morphism $\pi:A\to X$ 
whose Galois group $G$ is of order a power of $p$.
By Lemma~\ref{lem:et_cov}, we see that each $p^e$-torsion line bundle $\L$ on $A$ comes from $X$, 
so the action of $G$ preserves $\L$. 
Hence, by Proposition~\ref{prop:auto}, we can deduce $p^e=2$. 
\begin{small}
\begin{acknowledgement}
The first author wishes to express his gratitude to Professors Shunsuke Takagi and Zsolt Patakfalvi for fruitful discussions and valuable advice. He would like to thank Professor Nobuo Hara for effective comments and questions. Part of this work was carried out, with Preprint Seminar as a start, during his visit to Princeton University with support from The University of Tokyo/Princeton University Strategic Partnership Teaching and Research Collaboration Grant, and from the Program for Leading Graduate Schools, MEXT, Japan. He was also supported by JSPS KAKENHI Grant Number 15J09117.
The second author would like to thank Professor Shigeru Mukai for his warm encouragement and stimulating discussions. The second author was partially supported by
JSPS Grant-in-Aid for Young Scientists (B) 16K17581.
\end{acknowledgement}
\end{small}
\section{Notation}\label{section:notat}
Throughout this paper, we fix an algebraically closed field $k$ of characteristic $p>0$.
By a \textit{variety} we mean an integral separated scheme of finite type over $k$.
For a variety $X$, we denote by $F_X:X\to X$ the absolute Frobenius map of $X$. If there is no ambiguity, we write $F$ instead of $F_X$.
Let $m>0$ be an integer. 
A line bundle $\L$ on $X$ is said to be \textit{$m$-torsion} if $\L^m\cong\O_X$. 
A Cartier divisor $D$ on $X$ is said to be \textit{$m$-torsion} if $\O_X(D)$ is $m$-torsion. 
Let $\varphi:S\to T$ be a morphism between schemes. 
We denote by $\Pic(T)[\varphi]$ the kernel of the induced morphism $\varphi^*:\Pic(T)\to\Pic(S)$. 
Note that $\Pic(T)[F_T^e]$ forms the group of $p^e$-torsion line bundles on $T$. 
We often denote it by $\Pic(T)[p^e]$. 

\section{Line bundles trivialized by pullbacks}\label{section:lb}
Let $\pi:X\to Y$ be a surjective finite morphism of varieties. 
We first consider relationships between $\Pic(Y)[\pi]$ and the structure of $\pi_*\O_X$.
\begin{lem}\label{lem:decomp}
Let $\pi:X\to Y$ be a finite morphism between projective varieties such that 
the natural morphism $\O_Y\to \pi_*\O_X$ is injective and splits. 
Then $\#\Pic(Y)[\pi]\le\deg\pi$. Furthermore, if the equality holds 
then $\pi_*\O_X$ is isomorphic to the direct sum of all elements in $\Pic(Y)[\pi]$.
\end{lem}
\begin{proof}
For every $\L\in\Pic(Y)[\pi]$, by the projection formula, we have a morphism 
$$\L\cong \L\otimes\O_Y\to\L\otimes\pi_*\O_X\cong\pi_*\pi^*\L\cong\pi_*\O_X$$ that is injective and splits. 
Hence, we see that $\bigoplus_{\L\in\Pic(Y)[\pi]}\L$ is a direct summand of $\pi_*\O_X$ by the Krull--Schmidt theorem~\cite{Ati56}, 
so $\#\Pic(Y)[\pi]\le\mathrm{rank}\;\pi_*\O_X=\deg\pi$.
If the equality holds, then $(\pi_*\O_X)/(\bigoplus_{\L\in\Pic(Y)[\pi]}\L)$ is a torsion $\O_X$-submodule 
of the torsion free $\O_X$-module $\pi_*\O_X$, so it is zero. This is our assertion.
\end{proof}
\begin{prop}[\textup{\cite[Lemmas 3.2, 3.3 and 4.4]{ST16}}]\label{prop:equiv} 
Let $X$ be a smooth projective variety with pseudo-effective $K_X$ and $e$ be a positive integer. 
Then the following are equivalent: 
\begin{itemize} 
\item[(i)] ${F^e}_*\O_X$ is isomorphic to a direct sum of line bundles.  
\item[(ii)] $X$ is $F$-split and the number of $p^e$-torsion line bundles is equal to $p^{e\dim X}$.  
\end{itemize} 
\end{prop}
\begin{proof} 
The implication (i)$\Rightarrow$(ii) (resp. (ii)$\Rightarrow$(i)) follows from~\cite[Lemmas 3.2 and 4.4]{ST16} (resp.~\cite[Lemma 3.3]{ST16}). 
Note that $\deg{F^e}=p^{e\dim X}$.  
\end{proof}
\begin{rem}\label{rem:e to e-1}
Let $X$ be a smooth projective variety with pseudo-effective $K_X$. 
By Proposition~\ref{prop:equiv}, we see that if ${F^e}_*\O_X$ is  
a direct sum of line bundles for some $e>0$, then so is ${F^{d}}_*\O_X$ for each $0< d\le e$.
\end{rem}
For a smooth variety $X$, we denote the kernel and the image of the $\O_X$-morphism $d:F_*\Omega_X^i\to F_*\Omega_X^{i+1}$ by 
$Z_X^i$ and $B_X^{i+1}$, respectively. By the Cartier isomorphism, we then have
$$0\to B_X^i\to Z_X^i\xrightarrow{C}\Omega_X^i\to0.$$
\begin{prop}[\textup{\cite[Proposition 4.3]{GK03}}]\label{prop:GK}
Let $X$ be a smooth complete variety for which all global 1-forms are closed and such that $C$ gives a bijection 
$$H^0(X,Z_X^1)\to H^0(X,\Omega_X^1).$$ Then there exists an isomorphism $$H^0(X,\Omega_X^1)\xrightarrow{\cong}\Pic(X)[p]\otimes_{\Z}k.$$
\end{prop}
\begin{proof}
We refer the proof of~\cite[Proposition 4.3]{GK03}.
\end{proof}
Note that the assumptions of the proposition are satisfied when $X$ is $F$-split. 
Indeed, by the split exact sequence $$0\to \O_X\to F_*\O_X\xrightarrow{d} B_X^1\to0,$$ we have $H^j(X,B_X^1)=0$ for each $j\ge0$. 
Hence, by the exact sequence just before the proposition,  
we see that $H^0(X,Z_X^1)\xrightarrow{C} H^0(X,\Omega_X^1)$ is an isomorphism. 
In particular, $\dim H^0(X,Z_X^1)=\dim H^0(X,\Omega_X^1)$. 
The natural injection $H^0(X,Z_X^1)\to H^0(X,F_*\Omega_X^1)$ is then an isomorphism.
\section{Triviality of tangent bundles and $p$-torsion line bundles}\label{section:p-tors}
In this section, we investigate the relationship between 
$p$-torsion line bundles on a projective variety and the triviality of the tangent bundle. \par
Let $X$ be a normal projective variety and $J= \{L_1,\ldots,L_n\}$ be a set of $p$-torsion Cartier divisors on $X$. 
We define the finite morphism $\sigma_J:X_J\to X$ associated with $J$ as follows. 
Let $g_1,\ldots,g_n$ be rational functions of $X$ such that $pL_i=(g_i)$ for each $i$, 
where $(g_i)$ is the divisor of $g_i$.
We then take the normalization $X_J$ of $X$ in $K(X)(g_1^{p^{-1}},\ldots,g_n^{p^{-1}})$, 
and $\sigma_J:X_J\to X$ denotes the natural morphism. 
It is easily seen that $\deg \sigma_J\le p^n$. 
Furthermore, by the construction, we have a morphism $\tau_J:X\to X_J$ satisfying $F_X=\sigma_J\circ\tau_J:X\to X_J\to X$. 
This induces $$\Pic(X)[\sigma_J]\subseteq\Pic(X)[F_X]=\Pic(X)[p].$$
\begin{lem}\label{lem:basis}
With notation as above, 
assume that $\O_X(L_1),\ldots,\O_X(L_n)$ are linearly independent in the $\mathbb F_p$-vector space $\Pic(X)[p]$. 
Then they form a basis of the $\mathbb F_p$-vector space $\Pic(X)[\sigma_J]$.
\end{lem}
\begin{proof}
Set $d_J:=\deg\sigma_J$ and $\nu_J:=\#\Pic(X)[\sigma_J]$.
By the construction of $\sigma_J$, we have 
$${\sigma_J}^*L_i={\sigma_J}^*\frac{(g_i)}{p}=\frac{({\sigma_J}^*g_i)}{p}=\frac{((g_i^{p^{-1}})^p)}{p}=(g_i^{p^{-1}}),$$ 
and therefore ${\sigma_J}^*L_i\sim 0$ for each $L_i$. 
This implies that $$p^{\#J}=\#\{\O_X(\sum m_jL_j)|0\le m_j<p\}\le \nu_J.$$ 
Since $F_X=\sigma_J\circ\tau_J$ as mentioned above, we have the homomorphism 
$$ \O_{X}\to {\sigma_J}_*\O_{X_J}\to {F_X}_*\O_X. $$ 
The $F$-splitting of $X$ then induces a slitting of $\O_X\to{\sigma_J}_*\O_{X_J}$, 
so $\nu_J \le d_J$ as shown in Lemma~\ref{lem:decomp}.
Since $d_J\le p^{\#J}$, we deduce that $\nu_J=p^{\#J}$, 
which implies that $\Pic(X)[\sigma_J]$ is generated by 
$\O_X(L_1),\ldots, \O_X(L_n)$ as an $\mathbb F_p$-vector space.
\end{proof}
%
%
\begin{proof}[Proof of Theorem~\ref{thm:ch_triv_intro}.]
Set $n:=\dim X$. 
By Proposition~\ref{prop:equiv}, we see that $X$ is $F$-split and $\#\Pic(X)[p]=p^n$. 
Let $I=\{L_1,\ldots, L_n\}$ be a set of $p$-torsion divisors 
such that $\O_X(L_1),\ldots,\O_X(L_n)$ form a $\mathbb{F}_p$-basis of $\Pic(X)[p]$.  
Let $g_1,\ldots,g_n$ be rational functions of $X$ such that $pL_i=(g_i)$ for each $i$. 
%
Then, by a local calculation, we obtain that $dg_i/g_i\in H^0(X,\Omega_X^1)$.
We show that $dg_1/g_1,\ldots,dg_n/g_n$ form a basis of 
$K(X)$-vector space $(\Omega_X^1)_{\eta}\cong\Omega_{K(X)/k}^1$, 
where $\eta$ is the generic point of $X$. 
Suppose that there exists $J\subseteq I$ and $L_i\in I\setminus J$ such that $dg_i/g_i=\sum_{L_j\in J}a_jdg_j/g_j$ for some $a_j\in K(X)$.
Since there exists the $p$-th root of ${\sigma_J}^*g_j$ in $K(X_J)$ for each $L_j\in J$, we see that 
$$\frac{d({\sigma_J}^*g_i)}{{\sigma_J}^*g_i}={\sigma_J}^*\frac{dg_i}{g_i}={\sigma_J}^*\sum_{L_j\in J}\frac{a_jdg_j}{g_j}
=\sum_{L_j\in J}\frac{({\sigma_J}^*a_j)d({\sigma_J}^*g_j)}{{\sigma_J}^*g_j} =0.$$
Thus, $d({\sigma_J}^*g_i)=0$, or equivalently, there exists the $p$-th root of ${\sigma_J}^*g_i$ in $K(X_J)$. 
Hence, ${\sigma_J}^*L_i\sim0$, which contradicts Lemma~\ref{lem:basis}. 
Let $\varphi:\bigoplus_{1\le i\le n}\O_X\to\Omega_X^1$ be the morphism induced by $dg_1/g_1,\ldots,dg_n/g_n$. 
Taking the determinants, we obtain the morphism $\mathrm{det}(\varphi):\O_X\to \o_X$ which is an isomorphism over $\eta$, 
particularly it is injective. 
Since $\o_X$ is numerically trivial, we see that $\mathrm{det}(\varphi)$ is an isomorphism, 
which implies that $\varphi$ is also an isomorphism. 
This proves our assertion. 
\end{proof}
\begin{rem}
By Propositions \ref{prop:equiv} and \ref{prop:GK}, 
we see that the converse of Theorem \ref{thm:ch_triv_intro} holds when $X$ is $F$-split. 
\end{rem}
\section{Characterization of ordinary abelian varieties}
The aim of this section it to prove Theorem~\ref{thm:ch_ab_intro}. 
We first recall a result of Mehta and Srinivas. 
\begin{thm}[\textup{\cite[Theorem 1]{MS87}}]\label{thm:MS}
Let $X$ be a smooth projective $F$-split variety with trivial tangent bundle. 
Then there exists a Galois covering $\pi:A\to X$ of degree $p^m$ for an integer $m>0$ from an ordinary abelian variety $A$.
\end{thm}
\begin{rem}\label{rem:ord}
In the original statement of~\cite[Theorem 1]{MS87}, $X$ is assumed to be ordinary. As shown in~\cite[Lemma (1.1)]{MS87}, the ordinarity of $X$ is equivalent to the $F$-splitting of $X$, under the assumption that the tangent bundle of $X$ is trivial.
\end{rem}
%
The next lemma shows that $\pi:A\to X$ in Theorem~\ref{thm:MS} induces 
an isomorphism between the group of $p$-torsion line bundles on $A$ and that of $X$. 
\begin{lem}\label{lem:et_cov}
Let $e>0$ be an integer.  
Let $\pi:X\to Y$ be an \'etale morphism between smooth projective varieties. 
Assume that $Y$ satisfies the following conditions: 
\begin{itemize}
\item ${F_Y^e}_*\O_Y$ is isomorphic to a direct sum of line bundles. 
\item $K_Y$ is pseudo-effective. 
\end{itemize}
Then the induced homomorphism $\pi^*:\Pic(Y)[p^e]\to\Pic(X)[p^e]$ is an isomorphism.
\end{lem}
\begin{proof}
Since $\pi$ is \'etale, we have $\pi^*{F_Y^e}_*\O_Y\cong {F_X^e}_*{\pi^*}\O_Y\cong {F_X^e}_*\O_X$, 
and hence $X$ satisfies the same conditions as $Y$. 
Note that $K_X\sim\pi^*K_Y$, so $K_X$ is also pseudo-effective. 
By~\cite[Lemma 3.3 (1)]{ST16}, we see that ${F_X^e}_*\O_X$ and ${F_Y^e}_*\O_Y$ are respectively 
the direct sum of all $p^e$-torsion line bundles on $X$ and $Y$.
Then, we get the surjectivity of $\pi^*:\Pic(Y)[p^e]\to\Pic(X)[p^e]$.
The injectivity follows from~\cite[Lemma 3.3 (2)]{ST16}. 
\end{proof}
The following is the key proposition in the proof of Theorem~\ref{thm:ch_ab_intro}.
\begin{prop}\label{prop:auto}
Let $A$ be an ordinary abelian variety and $\sigma:A\to A$ be an isomorphism such that $\sigma^{p^d}$ is a translation for a positive integer $d$, but $\sigma^{p^{(d-1)}}$ is not a translation. Let $e$ be a positive integer. 
Assume that $\sigma^*\mathcal K\cong\mathcal K$ for each $p^e$-torsion line bundle $\mathcal K$ on $A$. Then $p^e=p^d=2$.
\end{prop}
We note that if $p=2$ and $\sigma=(-1)_A$, then $\sigma^p=1_A$ and $\sigma^*\mathcal K\cong\mathcal K^{-1}\cong\mathcal K$ for each $p$-torsion line bundle $\mathcal K$.
\begin{proof}
We denote by $t_a$ the translation on $A$ by $a$. 
Then, ${t_{a}}^*\L\cong\L$ for every $\L\in\Pic^0(X)$. 
Set $a:=\sigma(0)$ and $\rho:=t_{-a}\circ\sigma$. 
Then, $\rho$ is a nontrivial automorphism of abelian varieties, and for some $a'\in A$,
$$t_{a'}=\sigma^{p^d}=(t_{a}\circ\rho)^{p^d}
=t_{a}\circ t_{\rho(a)}\circ\cdots\circ t_{{\rho}^{p^d-1}(a)}\circ {\rho}^{p^d}
=t_{a+\rho(a)+\cdots+{\rho}^{p^d-1}(a)}\circ{\rho}^{p^d}.$$ 
Hence, $a'=a+\rho(a)+\cdots+{\rho}^{p^d-1}(a)$ and ${\rho}^{p^d}=\mathrm{id}_A$. 
Replacing $\sigma$ by $\rho$, we may assume that $\sigma$ is an automorphism of order $p^d$ of abelian varieties. \par
We show $e=1$. Let $B$ be an abelian subvariety of $A$ defined as the image of $1_A-\sigma^{p^{d-1}}$. 
Then, $\sigma$ induces an automorphism $\tau:=\sigma|_B$ on $B$.
By the choice of $B$, we see that $f(\sigma)|_B=f(\tau)$ for every $f(T)\in\Z[T]$.
Since $(\sum_{i=0}^{p-1}(\sigma^{p^{d-1}})^i)(1_A-\sigma^{p^{d-1}})=(1_A-(\sigma^{p^{d-1}})^p)=0_A,$ we have 
$\sum_{i=0}^{p-1}(\tau^{p^{d-1}})^i=(\sum_{i=0}^{p-1}(\sigma^{p^{d-1}})^i)|_B=0_B.$
Therefore, for each line bundle $\M$ on $B$ satisfying $\tau^*\M\cong \M$, we obtain
$$\M^p\cong \bigotimes_{i=0}^{p-1}((\tau^{p^{d-1}})^i)^*\M\cong (\sum_{i=0}^{p-1}(\tau^{p^{d-1}})^i)^*\M\cong (0_B)^*\M\cong \O_B.$$
We can take a torsion line bundle $\L$ of order $p^2$ on $A$ such that $\L|_B$ is also of order $p^2$.
Indeed, by the natural surjection ${F_A^2}_*\O_A\twoheadrightarrow{F_B^2}_*\O_B$, we obtain 
$$\bigoplus_{\L\in\Pic(A)[p^2]}\L|_B\cong({F_A^2}_*\O_A)|_B\twoheadrightarrow{F_B^2}_*\O_B\cong\bigoplus_{\M\in\Pic(B)[p^2]}\M.$$
This implies that the natural homomorphism $\Pic(A)[p^2]\to\Pic(B)[p^2]$ is surjective, because any nonzero homomorphism between numerically trivial line bundles on a projective variety is an isomorphism.
Note that since $B$ is a nontrivial ordinary abelian variety, there exists a torsion line bundle of order $p^2$ on $B$. 
Since $\L|_B^p\ncong\O_B$, we have $\tau^*(\L|_B)\ncong \L|_B$, and so $\sigma^*\L\ncong\L$. 
Hence, by the assumption, we see that $\L$ is not $p^e$-torsion, and thus $e=1$. \par
Next, we show $d=1$.
Let $\L$ be a $p^2$-torsion line bundle on $A$. Since $\L^p$ is $p$-torsion, we have $\sigma^*\L^p\cong\L^p$ by the assumption.
Set $\mathcal N:=\sigma^*\L\otimes\L^{-1}.$ 
Then, $\mathcal N^p\cong \sigma^*\L^p\otimes\L^{-p}\cong \L^p\otimes\L^{-p}\cong\O_A,$
so $\sigma^*\mathcal N\cong\mathcal N$.
From this, we obtain that for each $i\ge0$,
$$(\sigma^i)^*\L\cong (\sigma^{i-1})^*(\L\otimes\mathcal N)\cong((\sigma^{i-1})^*\L)\otimes\mathcal N\cong \cdots\cong\L\otimes\mathcal N^i.$$
Therefore, we have $(\sigma^p)^*\L\cong\L$ for every $p^2$-torsion line bundle $\L$. 
Thus, by the above argument, we see that $\sigma^p$ does not satisfy the assumption on $\sigma$ in the proposition. 
This means that $\sigma^p=0_A$, and so $d=1$.
\par
We finally show that $p=2$.
As shown before, we can take a torsion line bundle $\L$ on $A$ of order $p^2$ such that $\M:=\L|_B$ is also of order $p^2$. Set $\mathcal N:=\sigma^*\L\otimes\L^{-1}$. We then see that $\mathcal N^p\cong\O_A$ and $(\sigma^i)^*\L\cong\L\otimes\mathcal N^i$ for each $i\ge0$ as above, so
$$(\sum_{i=0}^{p-1}\sigma^i)^*\L\cong\bigotimes_{i=0}^{p-1}(\sigma^i)^*\L\cong\bigotimes_{i=0}^{p-1}(\L\otimes \mathcal N^i)=\L^p\otimes\mathcal N^{p(p-1)/2}.$$
Restricting the isomorphism to $B$, we obtain
$(\sum_{i=0}^{p-1}\tau^i)^*\M\cong \M^p\otimes(\mathcal N^{p(p-1)/2})|_B.$
Since $\sum_{i=0}^{p-1}\tau^i=0_B$, we see that the left hand side is trivial, so $(\mathcal N^{p(p-1)/2})|_B\cong \M^{-p}\ncong\O_B$.
This implies that $p(p-1)/2$ is not divisible by $p$, so $p=2$, which concludes the proof.
\end{proof}
%
%
\begin{proof}[Proof of Theorem~\ref{thm:ch_ab_intro}.] 
For an ordinary abelian variety $A$, we have $\#\Pic(A)[p^e]=p^{e\dim X}$ for each $e>0$, 
and hence the ``only if'' part follows from Proposition~\ref{prop:equiv}.
We show the ``if'' part. 
Let $e>0$ be an integer such that $F_*^e\O_X$ is a direct sum of line bundles. 
Note that ${F}_*\O_X$ is also a direct sum of line bundles, as mentioned in Remark~\ref{rem:e to e-1}.
By Theorem~\ref{thm:ch_triv_intro}, we see that $\Omega_X^1$ is trivial.
Then, by Theorem~\ref{thm:MS}, we obtain an ordinary abelian variety $A$ and 
a Galois \'etale cover $\pi:A\to X$ of degree $p^m$ for an integer $m>0$.
Let $G$ be the group of automorphisms $\sigma:A\to A$ of $A$ with $\pi\circ\sigma=\pi$. 
Note that $\#G=p^m$. 
Take a $\sigma\in G$. 
Let $d\ge0$ be the minimum integer such that $\sigma^{p^d}$ is a translation. 
By Lemma~\ref{lem:et_cov}, we have $\sigma^*\L\cong\L$ for each $\L\in\Pic(A)[p^e]$. 
If $d>0$, then we obtain $p^e=p^d=2$ by Proposition~\ref{prop:auto}, which contradicts the assumption.
Hence, every $\sigma\in G$ is a translation. 
Since $G$ acts freely on $A$, $G$ can be viewed as a subgroup of $A$, 
so the quotient $A/G\cong X$ is also an abelian variety (see, for example,~\cite[(4.40)]{GM}). 
\end{proof}
\begin{rem}
We assume that $F_*\O_X$ is a direct sum of line bundles, $K_X$ is pseudo-effective, and $p=2$.
By the same argument as above, we can show that $\sigma^p$ is a translation for each $\sigma\in G$. 
Hence, the set of translations $N$ in $G$ forms a normal subgroup of $G$ such that each element of $G/N$ is $p$-torsion. 
From this, we can give another proof of~\cite[Theorem 5.2]{Jos16}.
\end{rem}
\begin{rem}
Combining the above theorem with Theorem~\ref{thm:ch_triv_intro}, 
we can generalize a result of Li~\cite[Theorem 0.3]{Li10}, 
which states that an $F$-split projective variety of odd characteristic with trivial tangent bundle is an ordinary abelian variety
(see also~\cite[Theorem 5.2]{Jos16}). 
\end{rem}
%
%
%
\section{Case of curves and surfaces}\label{section:surf}
In this section, we classify all non-abelian smooth projective varieties $X$ of dimension at most 2 
with the property that $K_X$ is pseudo-effective and $F_*\O_X$ splits into a direct sum of line bundles. \par
In the case of curves, we obtain:
\begin{thm}
Let $X$ be a smooth projective curve. 
Then $X$ satisfies the following conditions if and only if $X$ is an ordinary elliptic curve:
\begin{itemize}
\item $F_*\O_X$ is isomorphic to a direct sum of line bundles.
\item $K_X$ is pseudo-effective.
\end{itemize}
\end{thm}
\begin{proof}
By Theorem~\ref{thm:ch_triv_intro}, 
we see that $X$ is $F$-split and $\Omega_X^1$ is trivial. 
This is equivalent to saying that $X$ is an ordinary elliptic curve.
\end{proof}
Next, we consider the case of surfaces.  
By Theorems~\ref{thm:ch_ab_intro} and~\ref{thm:ch_triv_intro}, 
we see that such surfaces exist only in characteristic 2 and have trivial tangent bundles.
Before stating the theorem, we recall the construction of an Igusa surface, 
the first example of variety with non-reduced Picard scheme ~\cite[2. The Example]{Igu55}.
Let $E_0$ and $E_1$ be two elliptic curves and $a\in E_0$ be a torsion point of order 2. 
Then, $\Z/2\Z$ acts on $E_0\times E_1$ by $(x,y)\mapsto (x+a,-y)$.
The quotient of this action is called an Igusa surface.
\begin{thm}\label{igusa} \samepage
Let $X$ be a non-abelian smooth projective surface. 
Then $X$ satisfies the following conditions if and only if 
$X$ is an $F$-split Igusa surface:
\begin{itemize}
\item $F_*\O_X$ is isomorphic to a direct sum of line bundles.
\item $K_X$ is pseudo-effective.
\end{itemize}
\end{thm}
\begin{proof}
As mentioned above, $T_X$ is trivial, so $K_X\sim0$.
By Theorem~\ref{thm:MS}, there exists a Galois \'etale cover $\pi:A\to X$ 
of degree $p^m$ for an $m>0$ from an ordinary abelian variety $A$.
Then, $\chi(\O_X)=p^{-m}\chi(\O_A)=0$.
By Bombieri and Mumford's classification of minimal surfaces of Kodaira dimension zero~\cite[p.25]{BM2}, 
we see that $X$ is either a hyperelliptic surface or a quasi-hyperelliptic surface. 
Let $f:X\to C$ be the Albanese map of $X$. 
We then have nonzero morphism $T_X\to f^*T_C$, and it is obviously surjective.
Therefore, $f$ is smooth and $X$ is a hyperelliptic surface.
Hence, we have an elliptic curves $E_0$, $E_1$ and 
a finite group scheme $G$ acting freely on $E_0\times E_1$ 
such that the quotient is isomorphic to $X$.
By the list of order of $K_X$ in~\cite[p.37]{BM2}, we see that $X$ is of type a).
Note that hyperelliptic surfaces of type c) also have trivial canonical bundles, but they are not $F$-split, 
because fibers of the Albanese maps are supersinguler elliptic curves. 
Since the base field is of characteristic 2, $X$ is of type a1) or a3). \par
We observe the case of a3). 
In this case, $G=(\Z/2\Z)\times\mu_2$. 
Here, $\Z/2\Z$ acts as $(x,y)\mapsto (x+a,-y)$ and $\mu_2$ acts by translation on both factors. 
Since $\mu_2$ is a normal subgroup scheme of $G$, 
we can take the quotient $\sigma:E_0\times E_1\to (E_0\times E_1)/\mu_2=:A$.
Then $\Z/2\Z$ acts freely on $A$ and $A/(\Z/2\Z)\cong X$. 
Recall that the quotient map of $E_i$ for the action of $\mu_2$ is the relative Frobenius map 
$F_{E_i/k}:=(F_{E_i},\nu_i):E_i\to E_i\times_k{\Spec k^{1/p}}$, 
where $\nu_i:E_i\to \mathrm{Spec}\;{k}$ is the structure morphism of $E_i$. 
Since the base field $k$ is algebraically closed, 
we may identify $F_{E_i/k}$ as the absolute Frobenius map $F_{E_i}$.
Hence, we obtain the following diagram:
$$\xymatrix@R=25pt@C=25pt{E_0 \ar[d]_{F_{E_0}} & E_0\times E_1 \ar[d]_{\alpha} \ar[r]^{\pi_1} \ar[l]_{\pi_0} & E_1 \ar[d]^{F_{E_1}} 
\\ E_0 & A \ar[r]^{\rho_1} \ar[l]_{\rho_0} & E_1. } $$
Considering the degrees of morphisms, we see that each square is cartesian. 
Let $\L_i$ be the torsion line bundle of order 2 on $E_i$. 
Since $\rho_i$ is flat, we have 
$$ \alpha_*\O_{E_0\times E_1} \cong {\rho_i}^*{F_{E_i}}_*\O_{E_i}\cong{\rho_i}^*(\O_{E_i}\oplus\L_i)\cong \O_A\oplus{\rho_i}^*\L_i,$$
so ${\rho_1}^*\L_1\cong{\rho_0}^*\L_0$. 
Let $\M_i$ be a line bundle on $E_i$ such that ${\M_i}^2\cong\L_i$. 
Let $\M:={\rho_0}^*\M_0\otimes{\rho_1}^*\M_1$.
Then $$ \M^2 \cong{\rho_0}^*{\M_0}^2\otimes{\rho_1}^*{\M_1}^2 \cong{\rho_0}^*\L_0\otimes{\rho_1}^*\L_1 \cong({\rho_0}^*\L_0)^2 \cong\O_A.  $$
We denote by $\sigma:A\to A$ the action of $\Z/2\Z$ on $A$.
Since ${(-1)_{E_1}}^*\M_1\cong {\M_1}^{-1}\cong {\M_1}^3\cong \L_1\otimes\M_1$, we have 
\begin{align*}
\sigma^*\M \cong \sigma^*{\rho_0}^*\M_0\otimes\sigma^*{\rho_1}^*{\M_1} \cong {\rho_0}^*\M_0\otimes{\rho_1}^*{\M_1}^3
\cong \M\otimes{\rho_1}^*\L_1 \ncong \M,
\end{align*}
which implies that $\M$ is not the pullback of a line bundle on $X\cong A/(\Z/2\Z)$. 
Hence, by Lemma~\ref{lem:et_cov}, we see that ${F_X}_*\O_X$ is not a direct sum of line bundles. \par
We consider the case of a1), or equivalently, the case when $X$ is an Igusa surface.
First, we show that $X$ is $F$-split. 
Consider the composite 
\begin{align*}
\O_{X}\to {F_X}_*\O_X \to \O_X \tag{\ref{section:surf}.1}
\end{align*}
of the Frobenius map and the canonical dual of the Frobenius map.
Since $Hom(\O_X, \O_X)=k$, it is enough to show that this composite is nonzero. 
Taking the pullback to $E_0\times E_1=A$, we get
\begin{align*}
\O_{A}\to {F_A}_*\O_A \to \O_A.  \tag{\ref{section:surf}.2}
\end{align*}
Since the pullback of $\omega_X$ is isomorphic to $\omega_A$, 
the second map of (\ref{section:surf}.2) is isomorphic to the canonical dual of the Frobenius map on $A$.
In particular, the composite (\ref{section:surf}.2) is nonzero, and hence so is (\ref{section:surf}.1). 
Second, we show that $\dim H^0(X,\Omega^1_X)=2$. 
Note that the action of $G$ on each coordinate is a translation or a multiplication by $-1$.
It is easy to see that both of the action is trivial on the cotangent bundle. 
Hence, $\Omega^1_X$ is trivial, and the assertion holds. 
Therefore, by Theorem~\ref{thm:ch_triv_intro}, 
we see that ${F_X}_*\O_X$ is a direct sum of line bundles, 
which completes the proof.
\end{proof}
\section{Higher-dimensional examples}\label{section:higher}
In this section, we give examples of higher-dimensional non-abelian varieties 
such that $K_X$ is pseudo-effective and $F_*\O_X$ splits into a direct sum of line bundles. 
The construction of each example is essentially based on the one of Igusa surfaces. 
Throughout this section, we assume that the base field is of characteristic 2. 
\begin{eg}
Let $A$ be an ordinary abelian variety of dimension $d\ge 1$. 
Then the group $A[2]$ of $2$-torsion points in $A$ is isomorphic to $(\Z/2\Z)^{\oplus d}$. 
Let $\{a_1,\ldots,a_d\}$ be a basis of the $\mathbb{F}_2$-vector space $A[2]$. 
Let $B$ be an ordinary abelian variety and $\rho_B:B\to B$ be a homomorphism. 
Assume that ${\rho_B}^2=1_B$ and $\rho_B$ acts trivially on $H^0(B,\Omega_B^1)$. 
For instance, if we set $\rho_B:=(-1)_B$, then these assumptions hold.
Let $c$ be an integer with $1\le c\le d$. 
We define the action of $G:=(\Z/2\Z)^{\oplus c}$ on $A\times B$ as follows. 
For every $g=(g_1,\ldots,g_c)\in G$, 
$$g\cdot(x,y):=(x+g_1a_1+\cdots+g_ca_c,{\rho_B}^{g_1+\cdots+g_c}(y)).$$
It is easily seen that this action is free, so the quotient $X:=(A\times B)/G$ is smooth.
By an argument similar to the case of a1) in the proof of Theorem~\ref{igusa}, we obtain the following:
\begin{itemize}
\item ${F_*}\O_X$ is isomorphic to a direct sum of line bundles.
\item $K_X$ is trivial.
\end{itemize}
Note that this construction is the same as that of Igusa surfaces in the case when $A$ and $B$ are elliptic curves and $\rho_B=(-1)_B$.
\end{eg}
The next example is one that has a Galois \'etale cover $\pi:A\to X$ satisfying the following conditions:
\begin{itemize}
\item $A$ is an abelian variety which is not isomorphic to a nontrivial direct product of abelian varieties.
\item There is no nontrivial factorization $\pi:A\to A'\xrightarrow{\pi'}X$ 
with Galois \'etale cover $\pi':A'\to X$ from an abelian variety $A'$.
\end{itemize}
\begin{eg}
Let $B$ and $C$ be simple ordinary abelian varieties. 
Assume that $\dim B\ge 2$. Then the number of torsion points of order 2 in $B$ is $2^{\dim B}-1\ge 2$. 
Let $b, b'\in B$ and $c\in C$ be distinct torsion points of order 2.
We define the action of $G':=(\Z/2\Z)\oplus(\Z/2\Z)$ on $B\times C$ as 
$$(1,0)\cdot(x,y):=(x+b,y+c) \textup{\quad and\quad } (0,1)\cdot(x,y):=(x+b',-c).$$
Clearly, this action is free. We set $H'$ to be the normal subgroup $\{(0,0), (1,0)\}$ of $G'$. 
Then $H'$ acts on $B\times C$ by translation, so the quotient $A:=(B\times C)/H'$ is also an ordinary abelian variety.
Note that since the degree of $B\times C\to A$ is 2, one can show that $A$ is not a nontrivial direct product of abelian varieties.
Let $X:=A/G.$ Then $X$ also can be written as $(B\times C)/G'$. By an argument similar to the case of a1) in the proof of Theorem~\ref{igusa}, we see that 
\begin{itemize}
\item ${F_*}\O_X$ is isomorphic to a direct sum of line bundles, and 
\item $K_X$ is trivial.
\end{itemize}
\end{eg}
\bibliographystyle{abbrv}
\bibliography{decomp}
\end{document}